\documentclass{amsart}
\usepackage{amssymb}
\usepackage{amsmath}
\usepackage{amsfonts}

\newtheorem{theorem}{Theorem}
\theoremstyle{plain}

\newtheorem{corollary}{Corollary}

\newtheorem{definition}{Definition}
\newtheorem{example}{Example}

\newtheorem{lemma}{Lemma}

\newtheorem{remark}{Remark}

\numberwithin{equation}{section}
\input{tcilatex}

\begin{document}
\title{ANTI-INVARIANT RIEMANNIAN SUBMERSIONS FROM SASAKIAN MANIFOLDS}
\author{I. K\"{u}peli Erken}
\address{Art and Science Faculty,Department of Mathematics, Uludag
University, 16059 Bursa, TURKEY}
\email{iremkupeli@uludag.edu.tr}
\author{C. Murathan}
\address{Art and Science Faculty,Department of Mathematics, Uludag
University, 16059 Bursa, TURKEY}
\email{cengiz@uludag.edu.tr}
\date{20.02.2013}
\keywords{Riemannian submersion, Sasakian manifold, Anti-invariant
submersion \ \ \ This paper is supported by Uludag University research
project (KUAP(F)-2012/57).}

\begin{abstract}
We introduce anti-invariant Riemannian submersions from Sasakian manifolds
onto Riemannian manifolds. We survey main results of anti-invariant
Riemannian submersions defined on Sasakian manifolds. We investigate
necessary and sufficient condition for an anti-invariant Riemannian
submersion to be totally geodesic and harmonic. We give examples of
anti-invariant submersions such that characteristic vector field $\xi $ is
vertical or horizontal.
\end{abstract}

\maketitle

\section{\textbf{Introduction}}

Let $F$ be a $C^{\infty }$-submersion from a Riemannian manifold ($M,g_{M})$
onto a Riemannian manifold $(N,g_{N}).$ Then according to the conditions on
the map $F:(M,g_{M})\rightarrow (N,g_{N}),$ we have the following
submersions:

semi-Riemannian submersion and Lorentzian submersion \cite{FAL}, Riemannian
submersion (\cite{BO1}, \cite{GRAY}), slant submersion (\cite{CHEN}, \cite%
{SAHIN1}), almost Hermitian submersion \cite{WAT}, contact-complex
submersion \cite{IANUS}, quaternionic submersion \cite{IANUS2}, almost $h$%
-slant submersion and $h$-slant submersion \cite{PARK1}, semi-invariant
submersion \cite{SAHIN2}, $h$-semi-invariant submersion \cite{PARK2}, etc.
As we know, Riemannian submersions are related with physics and have their
applications in the Yang-Mills theory (\cite{BL}, \cite{WATSON}),
Kaluza-Klein theory (\cite{BOL}, \cite{IV}), Supergravity and superstring
theories (\cite{IV2}, \cite{MUS}). In \cite{SAHIN}, Sahin introduced
anti-invariant Riemannian submersions from almost Hermitian manifolds onto
Riemannian manifolds. In this paper we consider anti-invariant Riemannian
submersions from Sasakian manifolds. The paper is organized as follows: In
section 2, we present the basic information about Riemannian submersions
needed for this paper. In section 3, we mention about Sasakian manifolds. In
section 4, we give definition of anti-invariant Riemannian submersions and
introduce anti-invariant Riemannian submersions from Sasakian manifolds onto
Riemannian manifolds. We survey main results of anti-invariant submersions
defined on Sasakian manifolds. We give examples of anti-invariant
submersions such that characteristic vector field $\xi $ is vertical or
horizontal.

\section{\textbf{Riemannian Submersions}}

In this section we recall several notions and results which will be needed
throughout the paper.

Let $(M,g_{M})$ be an $m$-dimensional Riemannian manifold , let $(N,g_{N})$
be an $n$-dimensional Riemannian manifold. A Riemannian submersion is a
smooth map $F:M\rightarrow N$ which is onto and satisfies the following
three axioms:

$S1$. $F$ has maximal rank.

$S2$. The differential $F_{\ast }$ preserves the lenghts of horizontal
vectors.

The fundamental tensors of a submersion were defined by O'Neill (\cite{BO1},%
\cite{BO2}). They are $(1,2)$-tensors on $M$, given by the formula:%
\begin{eqnarray}
\mathcal{T}(E,F) &=&\mathcal{T}_{E}F=\mathcal{H}\nabla _{\mathcal{V}E}%
\mathcal{V}F+\mathcal{V}\nabla _{\mathcal{V}E}\mathcal{H}F,  \label{AT1} \\
\mathcal{A}(E,F) &=&\mathcal{A}_{E}F=\mathcal{V}\nabla _{\mathcal{H}E}%
\mathcal{H}F+\mathcal{H}\nabla _{\mathcal{H}E}\mathcal{V}F,  \label{AT2}
\end{eqnarray}%
for any vector field $E$ and $F$ on $M.$ Here $\nabla $ denotes the
Levi-Civita connection of $(M,g_{M})$. These tensors are called
integrability tensors for the Riemannian submersions. Note that we denote
the projection morphism on the distributions ker$F_{\ast }$ and (ker$F_{\ast
})^{\perp }$ by $\mathcal{V}$ and $\mathcal{H},$ respectively. The following
\ Lemmas are well known (\cite{BO1},\cite{BO2}).

\begin{lemma}
For any $U,W$ vertical and $X,Y$ horizontal vector fields, the tensor fields 
$\mathcal{T}$, $\mathcal{A}$ satisfy:%
\begin{eqnarray}
i)\mathcal{T}_{U}W &=&\mathcal{T}_{W}U,  \label{TUW} \\
ii)\mathcal{A}_{X}Y &=&-\mathcal{A}_{Y}X=\frac{1}{2}\mathcal{V}\left[ X,Y%
\right] .  \label{TUW2}
\end{eqnarray}
\end{lemma}

It is easy to see that $\mathcal{T}$ $\ $is vertical, $\mathcal{T}_{E}=%
\mathcal{T}_{\mathcal{V}E}$ and $\mathcal{A}$ is horizontal, $\mathcal{A=A}_{%
\mathcal{H}E}$.

For each $q\in N,$ $F^{-1}(q)$ is an $(m-n)$ dimensional submanifold of $M$.
The submanifolds $F^{-1}(q),$ $q\in N,$ are called fibers. A vector field on 
$M$ is called vertical if it is always tangent to fibers. A vector field on $%
M$ is called horizontal if it is always orthogonal to fibers. A vector field 
$X$ on $M$ is called basic if $X$ is horizontal and $F$-related to a vector
field $X$ on $N,$ i. e., $F_{\ast }X_{p}=X_{\ast F(p)}$ for all $p\in M.$

\begin{lemma}
Let $F:(M,g_{M})\rightarrow (N,g_{N})$ be a Riemannian submersion. If $\ X,$ 
$Y$ are basic vector fields on $M$, then:
\end{lemma}

$i)$ $g_{M}(X,Y)=g_{N}(X_{\ast },Y_{\ast })\circ F,$

$ii)$ $\mathcal{H}[X,Y]$ is basic, $F$-related to $[X_{\ast },Y_{\ast }]$,

$iii)$ $\mathcal{H}(\nabla _{X}Y)$ is basic vector field corresponding to $%
\nabla _{X_{\ast }}^{^{\ast }}Y_{\ast }$ where $\nabla ^{\ast }$ is the
connection on $N.$

$iv)$ for any vertical vector field $V$, $[X,V]$ is vertical.

Moreover, if $X$ is basic and $U$ is vertical then $\mathcal{H}(\nabla
_{U}X)=\mathcal{H}(\nabla _{X}U)=\mathcal{A}_{X}U.$ On the other hand, from (%
\ref{AT1}) and (\ref{AT2}) we have%
\begin{eqnarray}
\nabla _{V}W &=&\mathcal{T}_{V}W+\hat{\nabla}_{V}W  \label{1} \\
\nabla _{V}X &=&\mathcal{H\nabla }_{V}X+\mathcal{T}_{V}X  \label{2} \\
\nabla _{X}V &=&\mathcal{A}_{X}V+\mathcal{V}\nabla _{X}V  \label{3} \\
\nabla _{X}Y &=&\mathcal{H\nabla }_{X}Y+\mathcal{A}_{X}Y  \label{4}
\end{eqnarray}%
for $X,Y\in \Gamma ((\ker F_{\ast })^{\bot })$ and $V,W\in \Gamma (\ker
F_{\ast }),$ where $\hat{\nabla}_{V}W=\mathcal{V}\nabla _{V}W.$

Notice that $\mathcal{T}$ \ acts on the fibres as the second fundamental
form of the submersion and restricted to vertical vector fields and it can
be easily seen that $\mathcal{T}=0$ is equivalent to the condition that the
fibres are totally geodesic. A Riemannian submersion is called a Riemannian
submersion with totally geodesic fiber if $\mathcal{T}$ $\ $vanishes
identically. Let $U_{1},...,U_{m-n}$ be an orthonormal frame of $\Gamma
(\ker F_{\ast }).$ Then the horizontal vector field $H$ $=\frac{1}{m-n}%
\dsum\limits_{j=1}^{m-n}\mathcal{T}_{U_{j}}U_{j}$ is called the mean
curvature vector field of the fiber. If \ $H$ $=0$ the Riemannian submersion
is said to be minimal. A Riemannian submersion is called a Riemannian
submersion with totally umbilical fibers if 
\begin{equation}
\mathcal{T}_{U}W=g_{M}(U,W)H  \label{4a}
\end{equation}%
for $U,W\in $ $\Gamma (\ker F_{\ast })$. For any $E\in \Gamma (TM),\mathcal{T%
}_{E\text{ }}$and $\mathcal{A}_{E}$ are skew-symmetric operators on $(\Gamma
(TM),g_{M})$ reversing the horizontal and the vertical distributions. By
Lemma 1 horizontally distribution $\mathcal{H}$ is integrable if and only if
\ $\mathcal{A=}0$. For any $D,E,G\in \Gamma (TM)$ one has%
\begin{equation}
g(\mathcal{T}_{D}E,G)+g(\mathcal{T}_{D}G,E)=0,  \label{4b}
\end{equation}%
\begin{equation}
g(\mathcal{A}_{D}E,G)+g(\mathcal{A}_{D}G,E)=0.  \label{4c}
\end{equation}

We recall the notion of harmonic maps between Riemannian manifolds. Let $%
(M,g_{M})$ and $(N,g_{N})$ be Riemannian manifolds and suppose that $\varphi
:M\rightarrow N$ is a smooth map between them. Then the differential $%
\varphi _{\ast }$ of $\varphi $ can be viewed a section of the bundle $\
Hom(TM,\varphi ^{-1}TN)\rightarrow M,$ where $\varphi ^{-1}TN$ is the
pullback bundle which has fibres $(\varphi ^{-1}TN)_{p}=T_{\varphi (p)}N,$ $%
p\in M.\ Hom(TM,\varphi ^{-1}TN)$ has a connection $\nabla $ induced from
the Levi-Civita connection $\nabla ^{M}$ and the pullback connection. Then
the second fundamental form of $\varphi $ is given by 
\begin{equation}
(\nabla \varphi _{\ast })(X,Y)=\nabla _{X}^{\varphi }\varphi _{\ast
}(Y)-\varphi _{\ast }(\nabla _{X}^{M}Y)  \label{5}
\end{equation}%
for $X,Y\in \Gamma (TM),$ where $\nabla ^{\varphi }$ is the pullback
connection. It is known that the second fundamental form is symmetric. If $%
\varphi $ is a Riemannian submersion it can be easily prove that 
\begin{equation}
(\nabla \varphi _{\ast })(X,Y)=0  \label{5a}
\end{equation}%
for $X,Y\in \Gamma ((\ker F_{\ast })^{\bot })$.A smooth map $\varphi
:(M,g_{M})\rightarrow (N,g_{N})$ is said to be harmonic if $trace(\nabla
\varphi _{\ast })=0.$ On the other hand, the tension field of $\varphi $ is
the section $\tau (\varphi )$ of $\Gamma (\varphi ^{-1}TN)$ defined by%
\begin{equation}
\tau (\varphi )=div\varphi _{\ast }=\sum_{i=1}^{m}(\nabla \varphi _{\ast
})(e_{i},e_{i}),  \label{6}
\end{equation}%
where $\left\{ e_{1},...,e_{m}\right\} $ is the orthonormal frame on $M$.
Then it follows that $\varphi $ is harmonic if and only if $\tau (\varphi
)=0 $, for details, \cite{B}.

\section{\textbf{Sasakian Manifolds}}

An $n$-dimensional differentiable manifold $M$ is said to have an almost
contact structure $(\phi ,\xi ,\eta )$ if it carries a tensor field $\phi $
of type $(1,1)$, a vector field $\xi $ and 1-form $\eta $ on $M$
respectively such that 
\begin{equation}
\phi ^{2}=-I+\eta \otimes \xi ,\text{ \ }\phi \xi =0,\text{ }\eta \circ \phi
=0,\text{ \ \ }\eta (\xi )=1,  \label{phi^2}
\end{equation}%
where $I$ denotes the identity tensor.

The almost contact structure is said to be normal if $N+d\eta \otimes \xi =0$%
, where $N$ is the Nijenhuis tensor of $\phi $. Suppose that a Riemannian
metric tensor $g$ is given in $M$ and satisfies the condition 
\begin{equation}
g(\phi X,\phi Y)=g(X,Y)-\eta (X)\eta (Y),\text{ \ \ }\eta (X)=g(X,\xi ).
\label{metric}
\end{equation}%
Then $(\phi ,\xi ,\eta ,g)$-structure is called an almost contact metric
structure. Define a tensor field $\Phi $ of type $(0,2)$ by $\Phi
(X,Y)=g(\phi X,Y)$. If $d\eta =\Phi $ then an almost contact metric
structure is said to be normal contact metric structure. A normal contact
metric structure is called a Sasakian structure, which satisfies 
\begin{equation}
(\nabla _{X}\phi )Y=g(X,Y)\xi -\eta (Y)X,  \label{Nambla fi}
\end{equation}%
where $\nabla $ denotes the Levi-Civita connection of $g$. For a Sasakian
manifold $M=M^{2n+1}$, it is known that 
\begin{equation}
R(\xi ,X)Y=g(X,Y)\xi -\eta (Y)X,  \label{R(xi,,X)Y}
\end{equation}%
\begin{equation}
S(X,\xi )=2n\eta (X)  \label{S(X,xi)}
\end{equation}%
and 
\begin{equation}
\nabla _{X}\xi =-\phi X.  \label{nablaXxi}
\end{equation}%
\cite{BL2}.

Now we will introduce a well known Sasakian manifold example on $%
\mathbb{R}
^{2n+1}.$

\begin{example}[\protect\cite{BL1}]
We consider $%
\mathbb{R}
^{2n+1}$ with Cartesian coordinates $(x_{i},y_{i},z)$ $(i=1,...,n)$ and its
usual contact form 
\begin{equation*}
\eta =\frac{1}{2}(dz-\dsum\limits_{i=1}^{n}y_{i}dx_{i}).
\end{equation*}%
The characteristic vector field $\xi $ is given by $2\frac{\partial }{%
\partial z}$ and its Riemannian metric $g$ and tensor field $\phi $ are
given by%
\begin{equation*}
g=\frac{1}{4}(\eta \otimes \eta
+\dsum\limits_{i=1}^{n}((dx_{i})^{2}+(dy_{i})^{2}),\text{ \ }\phi =\left( 
\begin{array}{ccc}
0 & \delta _{ij} & 0 \\ 
-\delta _{ij} & 0 & 0 \\ 
0 & y_{j} & 0%
\end{array}%
\right) \text{, \ }i=1,...,n
\end{equation*}%
This gives a contact metric structure on $%
\mathbb{R}
^{2n+1}$. The vector fields $E_{i}=2\frac{\partial }{\partial y_{i}},$ $%
E_{n+i}=2\left( \frac{\partial }{\partial x_{i}}+y_{i}\frac{\partial }{%
\partial z}\right) $, $\xi $ form a $\phi $-basis for the contact metric
structure. On the other hand, it can be shown that $%
\mathbb{R}
^{2n+1}(\phi ,\xi ,\eta ,g)$ is a Sasakian manifold.
\end{example}

\section{\textbf{Anti-invariant Riemannian submersions}}

\begin{definition}
Let $M(\phi ,\xi ,\eta ,g_{M})$ be a Sasakian manifold and $(N,g_{N})$ be a
Riemannian manifold. A Riemannian submersion $F:M(\phi ,\xi ,\eta
,g_{M})\rightarrow $ $(N,g_{N})$ is called an anti-invariant Riemannian
submersion if $\ker F_{\ast }$ is anti-invariant with respect to $\phi $,
i.e. $\phi (\ker F_{\ast })\subseteq (\ker F_{\ast })^{\bot }.$
\end{definition}

Let $F:M(\phi ,\xi ,\eta ,g_{M})\rightarrow $ $(N,g_{N})$ be an
anti-invariant Riemannian submersion from a Sasakian manifold $M(\phi ,\xi
,\eta ,g_{M})$ to a Riemannian manifold $(N,g_{N}).$ First of all, from
Definition 1, we have $\phi (\ker F_{\ast })\cap (\ker F_{\ast })^{\bot
}\neq \left\{ 0\right\} .$ We denote the complementary orthogonal
distribution to $\phi (\ker F_{\ast })$ in $(\ker F_{\ast })^{\bot }$ by $%
\mu .$ Then we have%
\begin{equation}
(\ker F_{\ast })^{\bot }=\phi \ker F_{\ast }\oplus \mu .  \label{A1}
\end{equation}%
Now we will introduce some examples.\ \ \ \ \ \ \ \ 

\begin{example}
$%
\mathbb{R}
^{5}$ \ has got a Sasakian structure as in Example 1. The Riemannian metric
tensor field $g_{%
\mathbb{R}
^{2}}$ is defined%
\begin{equation*}
g_{%
\mathbb{R}
^{2}}=\frac{1}{8}\left[ 
\begin{array}{cc}
1 & 0 \\ 
0 & 1%
\end{array}%
\right]
\end{equation*}%
on $%
\mathbb{R}
^{2}.$

Let $F:%
\mathbb{R}
^{5}\rightarrow 
\mathbb{R}
^{2}$ be a map defined by $%
F(x_{1},x_{2},y_{1},y_{2},z)=(x_{1}+y_{1},x_{2}+y_{2})$. Then, by direct
calculations 
\begin{equation*}
\ker F_{\ast }=span\{V_{1}=E_{1}-E_{3},\text{ }V_{2}=E_{2}-E_{4},\text{ }%
V_{3}=E_{5}=\xi \}
\end{equation*}%
and 
\begin{equation*}
(\ker F_{\ast })^{\bot }=span\{H_{1}=E_{1}+E_{3},\text{ }H_{2}=E_{2}+E_{4}\}.
\end{equation*}%
Then it is easy to see that $F$ is a Riemannian submersion. Moreover, $\phi
V_{1}=H_{1},$ $\phi V_{2}=H_{2},$ $\phi V_{3}=0$ imply that $\phi (\ker
F_{\ast })=(\ker F_{\ast })^{\bot }.$ As a result, $F$ is an anti-invariant
Riemannian submersion such that $\xi $ is vertical.
\end{example}

\begin{example}
Let $N$ be $%
\mathbb{R}
^{3}-\{(y_{1},y_{2},z)\in $ $%
\mathbb{R}
^{3}\mid y_{1}^{2}+y_{2}^{2}\leq 2\}$ and $%
\mathbb{R}
^{5}$ be a Sasakian manifold as in Example 1. The Riemannian metric tensor
field $g_{N}$ is given by 
\begin{equation*}
g_{N}=\frac{1}{4}\left[ 
\begin{array}{ccc}
\frac{1}{2} & \frac{y_{1}y_{2}}{2} & -\frac{y_{1}}{2} \\ 
\frac{y_{1}y_{2}}{2} & \frac{1}{2} & -\frac{y_{2}}{2} \\ 
-\frac{y_{1}}{2} & -\frac{y_{2}}{2} & 1%
\end{array}%
\right]
\end{equation*}%
on $N$.

Let $F:%
\mathbb{R}
^{5}\rightarrow N$ be a map defined by $%
F(x_{1},x_{2},y_{1},y_{2},z)=(x_{1}+y_{1},x_{2}+y_{2},\frac{y_{1}^{2}}{2}+%
\frac{y_{2}^{2}}{2}+z)$. \ After some calculations we have%
\begin{equation*}
\ker F_{\ast }=span\{V_{1}=E_{1}-E_{3},V_{2}=E_{2}-E_{4}\}
\end{equation*}%
and%
\begin{equation*}
(\ker F_{\ast })^{\bot }=span\{H_{1}=E_{1}+E_{3},\text{ }H_{2}=E_{2}+E_{4},%
\text{ }H_{3}=E_{5}=\xi \}
\end{equation*}%
Then it is easy to see that $F$ is a Riemannian submersion. Moreover, $\phi
V_{1}=H_{1}$, $\phi V_{2}=H_{2}$ imply that $\phi (\ker F_{\ast })\subset
(\ker F_{\ast })^{\bot }=$ $\phi (\ker F_{\ast })\oplus \{\xi \}$. Thus $F$
is an anti-invariant Riemannian submersion such that $\xi $ is horizontal.
\end{example}

\begin{example}
Let $N$ be $%
\mathbb{R}
^{4}-$ $\{(y_{1},y_{2},y_{3},z)\in $ $%
\mathbb{R}
^{4}\mid y_{1}^{2}+y_{2}^{2}+y_{3}^{2}-y_{1}^{2}y_{2}^{2}y_{3}^{2}\leq 2\}$%
and $%
\mathbb{R}
^{7}$ be a Sasakian manifold as in Example 1. The Riemannian metric tensor
field $g_{N}$ is given by 
\begin{equation*}
g_{N}=\frac{1}{4}\left[ 
\begin{array}{ccccc}
\frac{1}{2} & \frac{y_{1}y_{2}}{2} & -\frac{y_{1}}{2} & \frac{y_{1}y_{3}}{2}
& 0 \\ 
\frac{y_{1}y_{2}}{2} & \frac{1}{2} & -\frac{y_{2}}{2} & \frac{y_{2}y_{3}}{2}
& 0 \\ 
-\frac{y_{1}}{2} & -\frac{y_{2}}{2} & 1 & -\frac{y_{3}}{2} & 0 \\ 
\frac{y_{1}y_{3}}{2} & \frac{y_{2}y_{3}}{2} & -\frac{y_{3}}{2} & \frac{1}{2}
& 0 \\ 
0 & 0 & 0 & 0 & \frac{1}{2}%
\end{array}%
\right]
\end{equation*}%
on $N$.

Let $F:%
\mathbb{R}
^{7}\rightarrow N$ be a map defined by $%
F(x_{1},x_{2},x_{3},y_{1},y_{2},y_{3},z)=(x_{1}+y_{1},x_{2}+y_{2},\frac{%
y_{1}^{2}}{2}+\frac{y_{2}^{2}}{2}+\frac{y_{3}^{2}}{2}%
+z,x_{3}+y_{3},x_{3}-y_{3})$. \ After some calculations we have%
\begin{equation*}
\ker F_{\ast }=span\{V_{1}=E_{1}-E_{4},V_{2}=E_{2}-E_{5}\}
\end{equation*}%
and%
\begin{equation*}
(\ker F_{\ast })^{\bot }=span\{H_{1}=E_{1}+E_{4},\text{ }H_{2}=E_{2}+E_{5},%
\text{ }H_{3}=E_{3}-E_{6},H_{4}=E_{3}+E_{6},H_{5}=\xi =E_{7}\}
\end{equation*}%
Then it is easy to see that $F$ is a Riemannian submersion. $\phi
V_{1}=H_{1} $, $\phi V_{2}=H_{2}$ imply that $\phi (\ker F_{\ast })\subset
(\ker F_{\ast })^{\bot }=\phi (\ker F_{\ast })\oplus
span\{H_{3},H_{4},H_{5}\}.$ So $F$ is an anti-invariant Riemannian
submersion such that $\xi $ is horizontal.
\end{example}

\subsection{\textbf{Anti-invariant submersions admitting vertical structure
vector field }}

In this section, we will study anti-invariant submersions from a Sasakian
manifold onto a Riemannian manifold such that the characteristic vector
field $\xi $ is vertical.

It is easy to see that $\mu $ is an invariant distribution of $(\ker F_{\ast
})^{\bot },$ under the endomorphism $\phi .$ Thus, for $X\in \Gamma ((\ker
F_{\ast })^{\bot }),$ we write%
\begin{equation}
\phi X=BX+CX,  \label{A2}
\end{equation}%
where $BX\in \Gamma (\ker F_{\ast })$ and $CX\in \Gamma (\mu ).$ On the
other hand, since $F_{\ast }((\ker F_{\ast })^{\bot })=TN$ and $F$ is a
Riemannian submersion, using (\ref{A2}) we derive $g_{N}(F_{\ast }\phi
V,F_{\ast }CX)=0,$ for every $X\in $ $\Gamma ((\ker F_{\ast }))^{\perp \text{
}}$and $V\in \Gamma (\ker F_{\ast })$, which implies that 
\begin{equation}
TN=F_{\ast }(\phi (\ker F_{\ast }))\oplus F_{\ast }(\mu ).  \label{A2a}
\end{equation}

\begin{theorem}
Let $M(\phi ,\xi ,\eta ,g_{M})$ be a Sasakian manifold \ of dimension $2m+1$
and $(N,g_{N})$ is a Riemannian manifold of dimension $n$. Let $F:M(\phi
,\xi ,\eta ,g_{M})\rightarrow $ $(N,g_{N})$ be an anti-invariant Riemannian
submersion such that $\phi (\ker F_{\ast })=(\ker F_{\ast })^{\bot }$. Then
the characteristic vector field $\xi $ is vertical and $m=n$.
\end{theorem}

\begin{proof}
By the assumption $\phi (\ker F_{\ast })=(\ker F_{\ast })^{\bot }$, for any $%
U\in \Gamma (\ker F_{\ast })$ we have $g_{M}(\xi ,\phi U)=-g_{M}(\phi \xi
,U)=0$, which shows that the structure vector field is vertical. Now we
suppose that $U_{1},...,U_{k-1},\xi =U_{k}$ be an orthonormal frame of $%
\Gamma (\ker F_{\ast })$, where $k=2m-n+1$. Since $\phi (\ker F_{\ast
})=(\ker F_{\ast })^{\bot }$, $\phi U_{1},...,\phi U_{k-1}$ form an
orthonormal frame of $\Gamma ((\ker F_{\ast })^{\bot })$. So, by help of (%
\ref{A2a}) $\ $we obtain $k=n+1$ which implies that $m=n$.
\end{proof}

\begin{remark}
We note that Example 2 satisfies Theorem 1.
\end{remark}

\begin{theorem}
Let $M(\phi ,\xi ,\eta ,g_{M})$ be a Sasakian manifold \ of dimension $2m+1$
and $(N,g_{N})$ is a Riemannian manifold of dimension $n$. Let $F:M(\phi
,\xi ,\eta ,g_{M})\rightarrow $ $(N,g_{N})$ be an anti-invariant Riemannian
submersion. Then the fibers are not totally umbilical.
\end{theorem}

\begin{proof}
Using (\ref{1}) and (\ref{nablaXxi}) we obtain%
\begin{equation}
\mathcal{T}_{U}\xi =-\phi U  \label{T1}
\end{equation}%
for any $U\in \Gamma (\ker F_{\ast })$. If the fibers are totally umbilical,
then we have $\mathcal{T}_{U}V=g_{M}(U,V)H$ for any vertical vector fields $%
U,V$ where $H$ is the mean curvature vector field of any fibre. Since $%
\mathcal{T}_{\xi }\xi $ $=0$, we have $H=0$, which shows that fibres are
minimal. Hence the fibers are totally geodesic, which is a contradiction to
the fact that $\mathcal{T}_{U}\xi =-\phi U\neq 0$.
\end{proof}

From (\ref{phi^2}) and (\ref{A2}) we have following Lemma.

\begin{lemma}
Let $F$ be an anti-invariant Riemannian submersion from a Sasakian manifold $%
M(\phi ,\xi ,\eta ,g_{M})$ to a Riemannian manifold $(N,g_{N})$. Then we have%
\begin{eqnarray*}
BCX &=&0, \\
C^{2}X+\phi BX &=&-X,
\end{eqnarray*}%
for any $X\in \Gamma ((\ker F_{\ast })^{\bot }).$
\end{lemma}

Using (\ref{Nambla fi}) \ one can easily obtain 
\begin{equation}
\nabla _{X}Y=-\phi \nabla _{X}\phi Y+g(Y,\phi X)\xi  \label{Namblafi2}
\end{equation}%
for $X,Y\in \Gamma ((\ker F_{\ast })^{\bot }).$

\begin{lemma}
Let $F$ be an anti-invariant Riemannian submersion from a Sasakian manifold $%
M(\phi ,\xi ,\eta ,g_{M})$ to a Riemannian manifold $(N,g_{N})$. Then we have

\begin{equation}
CX=-\mathcal{A}_{X}\xi ,  \label{C1}
\end{equation}%
\begin{equation}
g_{M}(\mathcal{A}_{X}\xi ,\phi U)=0,  \label{C2}
\end{equation}%
\begin{equation}
g_{M}(\nabla _{Y}\mathcal{A}_{X}\xi ,\phi U)=-g_{M}(\mathcal{A}_{X}\xi ,\phi 
\mathcal{A}_{Y}U)+\eta (U)g_{M}(\mathcal{A}_{X}\xi ,Y)  \label{C3}
\end{equation}%
and%
\begin{equation}
g_{M}(X,\mathcal{A}_{Y}\xi )=-g_{M}(Y,\mathcal{A}_{X}\xi )  \label{C4}
\end{equation}%
for $X,Y\in \Gamma ((\ker F_{\ast })^{\bot })$ and $U\in \Gamma (\ker
F_{\ast }).$
\end{lemma}

\begin{proof}
By virtue of (\ref{3}) and (\ref{nablaXxi}) we have (\ref{C1}).

For $X\in \Gamma ((\ker F_{\ast })^{\bot })$ and $U\in \Gamma (\ker F_{\ast
})$, by virtue of (\ref{metric}), (\ref{A2}) and (\ref{C1}) we get%
\begin{eqnarray}
g_{M}(\mathcal{A}_{X}\xi ,\phi U) &=&-g_{M}(\phi X-BX,\phi U)  \label{C5} \\
&=&-g_{M}(X,U)+\eta (X)\eta (U)-g_{M}(\phi BX,U).  \notag
\end{eqnarray}%
Since $\phi BX\in \Gamma ((\ker F_{\ast })^{\bot })$ and $\xi \in \Gamma
(\ker F_{\ast }),$ (\ref{C5}) implies (\ref{C2}).

Now from (\ref{C2}) we get%
\begin{equation*}
g_{M}(\nabla _{Y}\mathcal{A}_{X}\xi ,\phi U)=-g_{M}(\mathcal{A}_{X}\xi
,\nabla _{Y}\phi U)
\end{equation*}%
for $X,Y\in \Gamma ((\ker F_{\ast })^{\bot })$ and $U\in \Gamma (\ker
F_{\ast })$. Then using (\ref{3}) and (\ref{Nambla fi}) we have 
\begin{equation*}
g_{M}(\nabla _{Y}\mathcal{A}_{X}\xi ,\phi U)=-g_{M}(\mathcal{A}_{X}\xi ,\phi 
\mathcal{A}_{Y}U)-g_{M}(\mathcal{A}_{X}\xi ,\phi (\mathcal{V}\nabla
_{Y}U))+\eta (U)g_{M}(\mathcal{A}_{X}\xi ,Y).
\end{equation*}%
Since $\phi (\mathcal{V}\nabla _{Y}U)\in \Gamma (\phi \ker F_{\ast })=\Gamma
((\ker F_{\ast })^{\bot }),$ we obtain (\ref{C3}).

Using (\ref{4c}), we obtain directly (\ref{C4}).
\end{proof}

We now study the integrability of the distribution $(\ker F_{\ast })^{\bot }$
and then we investigate the geometry of leaves of $\ker F_{\ast }$ and $%
(\ker F_{\ast })^{\bot }.$ We note that it is known that the distribution $%
\ker F_{\ast }$ \ is integrable.

\begin{theorem}
Let F be an anti-invariant Riemannian submersion from a Sasakian manifold $%
M(\phi ,\xi ,\eta ,g_{M})$ to a Riemannian manifold $(N,g_{N})$. Then the
following assertions are equivalent to each other;
\end{theorem}

$\ i)$ $(\ker F_{\ast })^{\bot }$ \textit{is integrable. }

$ii)$ 
\begin{eqnarray*}
g_{N}((\nabla F_{\ast })(Y,BX),F_{\ast }\phi V) &=&g_{N}((\nabla F_{\ast
})(X,BY),F_{\ast }\phi V) \\
&&+g_{M}(\mathcal{A}_{X}\xi ,\phi \mathcal{A}_{Y}V)-g_{M}(\mathcal{A}_{Y}\xi
,\phi \mathcal{A}_{X}V).
\end{eqnarray*}

$iii)$%
\begin{equation*}
g_{M}(\mathcal{A}_{X}BY-\mathcal{A}_{Y}BX,\phi V)=g_{M}(\mathcal{A}_{X}\xi
,\phi \mathcal{A}_{Y}V)-g_{M}(\mathcal{A}_{Y}\xi ,\phi \mathcal{A}_{X}V)
\end{equation*}%
\textit{for }$X,Y\in \Gamma ((\ker F_{\ast })^{\bot })$\textit{\ and }$V\in
\Gamma (\ker F_{\ast }).$

\begin{proof}
Using (\ref{Namblafi2}), for $X,Y\in \Gamma ((\ker F_{\ast })^{\bot })$ and $%
V\in \Gamma (\ker F_{\ast })$ we get%
\begin{eqnarray*}
g_{M}(\left[ X,Y\right] ,V) &=&g_{M}(\nabla _{X}Y,V)-g_{M}(\nabla _{Y}X,V) \\
&=&g_{M}(\nabla _{X}\phi Y,\phi V)-g_{M}(\nabla _{Y}\phi X,\phi
V)+2g_{M}(\phi X,Y)g_{M}(V,\xi ).
\end{eqnarray*}%
Then from (\ref{A2}) we have%
\begin{eqnarray*}
g_{M}(\left[ X,Y\right] ,V) &=&g_{M}(\nabla _{X}BY,\phi V)-g_{M}(\nabla _{X}%
\mathcal{A}_{Y}\xi ,\phi V)-g_{M}(\nabla _{Y}BX,\phi V) \\
&&+g_{M}(\nabla _{Y}\mathcal{A}_{X}\xi ,\phi V)+2g_{M}(\phi X,Y)g_{M}(V,\xi
).
\end{eqnarray*}%
Using (\ref{AT2}), (\ref{3}) and if we take into account that $F$ is a
Riemannian submersion, we obtain%
\begin{eqnarray*}
g_{M}(\left[ X,Y\right] ,V) &=&g_{N}(F_{\ast }\nabla _{X}BY,F_{\ast }\phi
V)-g_{M}(\nabla _{X}\mathcal{A}_{Y}\xi ,\phi V) \\
&&-g_{N}(F_{\ast }\nabla _{Y}BX,F_{\ast }\phi V)+g_{M}(\nabla _{Y}\mathcal{A}%
_{X}\xi ,\phi V) \\
&&-2g_{M}(\mathcal{A}_{X}\xi ,Y)g_{M}(V,\xi ).
\end{eqnarray*}%
Thus, from (\ref{5}) and (\ref{C3}) we have 
\begin{eqnarray*}
g_{M}(\left[ X,Y\right] ,V) &=&g_{N}(-(\nabla F_{\ast })(X,BY)+(\nabla
F_{\ast })(Y,BX),F_{\ast }\phi V) \\
&&+g_{M}(\mathcal{A}_{Y}\xi ,\phi \mathcal{A}_{X}V)-g_{M}(\mathcal{A}_{X}\xi
,\phi \mathcal{A}_{Y}V)
\end{eqnarray*}%
which proves $(i)\Leftrightarrow (ii).$ On the other hand using (\ref{5}) we
get%
\begin{equation*}
(\nabla F_{\ast })(Y,BX)-(\nabla F_{\ast })(X,BY)=-F_{\ast }(\nabla
_{Y}BX-\nabla _{X}BY).
\end{equation*}%
Then (\ref{3}) implies that 
\begin{equation*}
(\nabla F_{\ast })(Y,BX)-(\nabla F_{\ast })(X,BY)=-F_{\ast }(\mathcal{A}%
_{Y}BX-\mathcal{A}_{X}BY).
\end{equation*}%
From (\ref{AT2}) $\mathcal{A}_{Y}BX-\mathcal{A}_{X}BY\in \Gamma ((\ker
F_{\ast })^{\bot }),$ this shows that $(ii)\Leftrightarrow (iii).$
\end{proof}

\begin{remark}
If $\phi (\ker F_{\ast })=(\ker F_{\ast })^{\bot }$ then we get $C=0$ and
morever (\ref{A2a}) implies that $TN=F_{\ast }(\phi (\ker F_{\ast }))$.
\end{remark}

Hence we have the following Corollary.

\begin{corollary}
Let $F:M(\phi ,\xi ,\eta ,g_{M})\rightarrow $ $(N,g_{N})$ be an
anti-invariant Riemannian submersion such that $\phi (\ker F_{\ast })=(\ker
F_{\ast })^{\bot },$ where $M(\phi ,\xi ,\eta ,g_{M})$ is a Sasakian
manifold and $(N,g_{N})$ is a Riemannian manifold. Then following assertions
are equivalent to each other;

$i)(\ker F_{\ast })^{\bot }$ is\ integrable\textit{.}

$ii)(\nabla F_{\ast })(Y,\phi X)=(\nabla F_{\ast })(X,\phi Y)$ for $X,Y\in
\Gamma ((\ker F_{\ast })^{\bot }).$

$iii)\mathcal{A}_{X}\phi Y=\mathcal{A}_{Y}\phi X.$
\end{corollary}

\begin{theorem}
Let $F$ be an anti-invariant Riemannian submersion from a Sasakian manifold $%
M(\phi ,\xi ,\eta ,g_{M})$ to a Riemannian manifold $(N,g_{N}).$ Then the
following assertions are equivalent to each other;
\end{theorem}

$\ i)$ $(\ker F_{\ast })^{\bot }$ \textit{defines a totally geodesic
foliation on }$M.$

$ii)$ 
\begin{equation*}
g_{M}(\mathcal{A}_{X}BY,\phi V)=-g_{M}(\mathcal{A}_{Y}\xi ,\phi \mathcal{A}%
_{X}V).
\end{equation*}

$\bigskip iii)$%
\begin{equation*}
g_{N}((\nabla F_{\ast })(X,\phi Y),F_{\ast }\phi V)=g_{M}(\mathcal{A}_{Y}\xi
,\phi \mathcal{A}_{X}V)-g_{M}(\mathcal{A}_{Y}\xi ,X)\eta (V)
\end{equation*}%
\textit{for }$X,Y\in \Gamma ((\ker F_{\ast })^{\bot })$\textit{\ and }$V\in
\Gamma (\ker F_{\ast }).$

\begin{proof}
From (\ref{3}), (\ref{A2}), (\ref{Namblafi2}) and (\ref{C3}) we obtain%
\begin{equation}
g_{M}(\nabla _{X}Y,V)=g_{M}(\mathcal{A}_{X}BY,\phi V)+g_{M}(\mathcal{A}%
_{Y}\xi ,\phi \mathcal{A}_{X}V)-(g_{M}(\mathcal{A}_{Y}\xi ,X)+g_{M}(\mathcal{%
A}_{X}\xi ,Y))\eta (V)  \label{C6}
\end{equation}%
\textit{for }$X,Y\in \Gamma ((\ker F_{\ast })^{\bot })$\textit{\ and }$V\in
\Gamma (\ker F_{\ast }).$ Using (\ref{C4}) in (\ref{C6}) we get%
\begin{equation*}
g_{M}(\nabla _{X}Y,V)=g_{M}(\mathcal{A}_{X}BY,\phi V)+g_{M}(\mathcal{A}%
_{Y}\xi ,\phi \mathcal{A}_{X}V)
\end{equation*}%
The last equation shows $(i)\Leftrightarrow (ii)$.

For $X,Y\in \Gamma ((\ker F_{\ast })^{\bot })$\textit{\ and }$V\in \Gamma
(\ker F_{\ast }),$%
\begin{eqnarray}
g_{M}(\mathcal{A}_{X}BY,\phi V) &=&-g_{M}(\mathcal{A}_{Y}\xi ,\phi \mathcal{A%
}_{X}V)  \notag \\
&&\overset{(\ref{C3})}{=}g_{M}(\nabla _{X}\mathcal{A}_{Y}\xi ,\phi
V)-g_{M}(X,\mathcal{A}_{Y}\xi )\eta (V)  \notag \\
&&\overset{(\ref{A2})}{=}-g_{M}(\nabla _{X}\phi Y,\phi V)+g_{M}(\nabla
_{X}BY,\phi V)-g_{M}(X,\mathcal{A}_{Y}\xi )\eta (V)  \label{B6}
\end{eqnarray}%
Since differential $F_{\ast }$ preserves the lenghts of horizontal vectors
the relation (\ref{B6}) forms%
\begin{equation}
g_{M}(\mathcal{A}_{X}BY,\phi V)=g_{N}(F_{\ast }\nabla _{X}BY,F_{\ast }\phi
V)-g_{M}(\nabla _{X}\phi Y,\phi V)-g_{M}(X,\mathcal{A}_{Y}\xi )\eta (V)
\label{B7}
\end{equation}%
Using (\ref{Namblafi2}), (\ref{metric}), (\ref{5}) and (\ref{5a}) in (\ref%
{B7}) respectively, we obtain%
\begin{equation*}
g_{M}(\mathcal{A}_{X}BY,\phi V)=g_{N}(-(\nabla F_{\ast })(X,\phi Y),F_{\ast
}\phi V)-g_{M}(X,\mathcal{A}_{Y}\xi )\eta (V)
\end{equation*}%
which tells that $(ii)\Leftrightarrow (iii).$
\end{proof}

\begin{corollary}
Let $F:M(\phi ,\xi ,\eta ,g_{M})\rightarrow $ $(N,g_{N})$ be an
anti-invariant Riemannian submersion such that $\phi (\ker F_{\ast })=(\ker
F_{\ast })^{\bot }.$ Then the following assertions are equivalent to each
other;

$\ i)$ $(\ker F_{\ast })^{\bot }$ \textit{defines a totally geodesic
foliation on }$M.$

$ii)$ $\mathcal{A}_{X}\phi Y=0.$

$iii)$ $(\nabla F_{\ast })(X,\phi Y)=0$ \textit{for }$X,Y\in \Gamma ((\ker
F_{\ast })^{\bot })$\textit{\ and }$V\in \Gamma (\ker F_{\ast })$.
\end{corollary}

We note that a differentiable map $F$ between two Riemannian manifolds is
called totally geodesic if $\nabla F_{\ast }=0.$ Using Theorem 2 one can
easily prove that the fibers are not totally geodesic. Hence we have the
following Theorem.

\begin{theorem}
Let $F:M(\phi ,\xi ,\eta ,g_{M})\rightarrow $ $(N,g_{N})$ be an
anti-invariant Riemannian submersion where $M(\phi ,\xi ,\eta ,g_{M})$ is a
Sasakian manifold and $(N,g_{N})$ is a Riemannian manifold. Then $F$ is not
totally geodesic map.
\end{theorem}

Finally, we give a necessary and sufficient condition for an anti-invariant
Riemannian submersion such that $\phi (\ker F_{\ast })=(\ker F_{\ast
})^{\bot }$ to be harmonic.

\begin{theorem}
Let $F:M(\phi ,\xi ,\eta ,g_{M})\rightarrow $ $(N,g_{N})$ be an
anti-invariant Riemannian submersion such that $\phi (\ker F_{\ast })=(\ker
F_{\ast })^{\bot },$ where $M(\phi ,\xi ,\eta ,g_{M})$ is a Sasakian
manifold and $(N,g_{N})$ is a Riemannian manifold. Then $F$ is harmonic if
and only if Trace$\phi \mathcal{T}_{V}=(2m-n)\eta (V)$ for $V\in \Gamma
(\ker F_{\ast })$.
\end{theorem}

\begin{proof}
From \cite{EJ} we know that $F$ is harmonic if and only if $F$ has minimal
fibres. Thus $F$ is harmonic if and only if $\dsum\limits_{i=1}^{k}\mathcal{T%
}_{e_{i}}e_{i}=0,$ where $k=2m+1-n$ is dimension of $\ker F_{\ast }$ $.$ On
the other hand, from (\ref{1}), (\ref{2}) and (\ref{Nambla fi}) we get%
\begin{equation}
\mathcal{T}_{V}\phi W=\phi \mathcal{T}_{V}W-\eta (W)V+g(V,W)\xi  \label{S15}
\end{equation}%
for any $W,$ $V\in \Gamma (\ker F_{\ast }).$ Using (\ref{S15}), we get 
\begin{equation*}
\dsum\limits_{i=1}^{k}g_{M}(\mathcal{T}_{e_{i}}\phi
e_{i},V)=-\dsum\limits_{i=1}^{k}g_{M}(\mathcal{T}_{e_{i}}e_{i},\phi
V)+(k-1)\eta (V)
\end{equation*}%
for any $V\in \Gamma (\ker F_{\ast }).$ (\ref{4b}) implies that%
\begin{equation*}
\dsum\limits_{i=1}^{k}g_{M}(\phi e_{i},\mathcal{T}_{e_{i}}V)=\dsum%
\limits_{i=1}^{k}g_{M}(\mathcal{T}_{e_{i}}e_{i},\phi V)-(k-1)\eta (V)
\end{equation*}%
Then, using (\ref{TUW}) we have%
\begin{equation*}
\dsum\limits_{i=1}^{k}g_{M}(\phi e_{i},\mathcal{T}_{V\text{ }%
}e_{i})=\dsum\limits_{i=1}^{k}g_{M}(\mathcal{T}_{e_{i}}e_{i},\phi
V)-(k-1)\eta (V).
\end{equation*}%
Hence, proof comes from (\ref{metric}).
\end{proof}

\subsection{\textbf{Anti-invariant submersions admitting horizontal
structure vector field }}

In this section, we will study anti-invariant submersions from a Sasakian
manifold onto a Riemannian manifold such that the characteristic vector
field $\xi $ is horizontal. Using (\ref{A1}), we have $\mu =\phi \mu \oplus
\{\xi \}.$ For any horizontal vector field $X$ we put 
\begin{equation}
\phi X=BX+CX,  \label{IREM}
\end{equation}%
where where $BX\in \Gamma (\ker F_{\ast })$ and $CX\in \Gamma (\mu ).$

Now we suppose that $V$ is vertical and $X$ is horizontal vector field.
Using above relation and (\ref{metric}) we obtain%
\begin{equation*}
g_{M}(\phi V,CX)=0.
\end{equation*}%
From this last relation we have $g_{N}(F_{\ast }\phi V,F_{\ast }CX)=0$ which
implies that 
\begin{equation}
TN=F_{\ast }(\phi (\ker F_{\ast }))\oplus F_{\ast }(\mu ).  \label{Ba}
\end{equation}

\begin{theorem}
Let $M(\phi ,\xi ,\eta ,g_{M})$ be a Sasakian manifold \ of dimension $2m+1$
and $(N,g_{N})$ is a Riemannian manifold of dimension $n.$ Let $F:M(\phi
,\xi ,\eta ,g_{M})\rightarrow $ $(N,g_{N})$ be an anti-invariant Riemannian
submersion such that $(\ker F_{\ast })^{\bot }=\phi \ker F_{\ast }\oplus
\{\xi \}.$Then $m+1=n$.
\end{theorem}

\begin{proof}
We assume that $U_{1},...,U_{k\text{ }}$be an orthonormal frame of $\Gamma
(\ker F_{\ast })$, where $k=2m-n+1$. Since $(\ker F_{\ast })^{\bot }=\phi
\ker F_{\ast }\oplus \{\xi \}$, $\phi U_{1},...,\phi U_{k},\xi $ form an
orthonormal frame of $\Gamma ((\ker F_{\ast })^{\bot })$. So, by help of (%
\ref{A2a}) $\ $we obtain $k=n-1$ which implies that $m+1=n$.
\end{proof}

\begin{remark}
We note that Example 3 satisfies Theorem 7.
\end{remark}

From (\ref{phi^2}) and (\ref{Ba}) we obtain following Lemma.

\begin{lemma}
Let $F$ be an anti-invariant Riemannian submersion from a Sasakian manifold $%
M(\phi ,\xi ,\eta ,g_{M})$ to a Riemannian manifold $(N,g_{N})$. Then we have%
\begin{eqnarray*}
BCX &=&0, \\
\phi ^{2}X &=&C^{2}X+\phi BX
\end{eqnarray*}%
for $X\in \Gamma ((\ker F_{\ast })^{\bot }).$
\end{lemma}

Using (\ref{Nambla fi})\ one can easily obtain 
\begin{equation}
\nabla _{X}Y=-\phi \nabla _{X}\phi Y+\eta (\nabla _{X}Y)\xi -\eta (Y)\phi X
\label{Namblafi3}
\end{equation}%
for $X,Y\in \Gamma ((\ker F_{\ast })^{\bot }).$

\begin{lemma}
Let $F$ be an anti-invariant Riemannian submersion from a Sasakian manifold $%
M(\phi ,\xi ,\eta ,g_{M})$ to a Riemannian manifold $(N,g_{N})$. Then we have%
\begin{equation}
BX=-\mathcal{A}_{X}\xi ,  \label{IKE1}
\end{equation}%
\begin{equation}
\mathcal{T}_{U}\xi =0,  \label{IKE2}
\end{equation}%
\begin{equation}
g_{M}(\mathcal{A}_{X}\xi ,\phi U)=0,  \label{IKE3}
\end{equation}%
\begin{equation}
g_{M}(\nabla _{Y}\mathcal{A}_{X}\xi ,\phi U)=-g_{M}(\mathcal{A}_{X}\xi ,\phi 
\mathcal{A}_{Y}U),  \label{IKE4}
\end{equation}%
\begin{equation}
g_{M}(\nabla _{X}CY,\phi U)=-g_{M}(CY,\phi \mathcal{A}_{X}U),  \label{IKE6}
\end{equation}%
for $X,Y\in \Gamma ((\ker F_{\ast })^{\bot })$ and $U\in \Gamma (\ker
F_{\ast }).$
\end{lemma}

\begin{proof}
By virtue of (\ref{4}), (\ref{nablaXxi}) and (\ref{IREM}) we have (\ref{IKE1}%
). Using (\ref{2}) and (\ref{nablaXxi}) we obtain (\ref{IKE2}). Since $%
\mathcal{A}_{X}\xi $ is vertical and $\phi U$ is horizontal for $X\in \Gamma
((\ker F_{\ast })^{\bot })$ and $U\in \Gamma (\ker F_{\ast })$, we have (\ref%
{IKE3}). Now using (\ref{IKE3}) we get%
\begin{equation*}
g_{M}(\nabla _{Y}\mathcal{A}_{X}\xi ,\phi U)=-g_{M}(\mathcal{A}_{X}\xi
,\nabla _{Y}\phi U)
\end{equation*}%
for $X,Y\in \Gamma ((\ker F_{\ast })^{\bot })$ and $U\in \Gamma (\ker
F_{\ast })$. Then using (\ref{3}) and (\ref{Nambla fi}) we have 
\begin{equation*}
g_{M}(\nabla _{Y}\mathcal{A}_{X}\xi ,\phi U)=-g_{M}(\mathcal{A}_{X}\xi ,\phi 
\mathcal{A}_{Y}U)-g_{M}(\mathcal{A}_{X}\xi ,\phi (\mathcal{V}\nabla _{Y}U)).
\end{equation*}%
Since $\phi (\mathcal{V}\nabla _{Y}U)\in \Gamma ((\ker F_{\ast })^{\bot }),$
we obtain (\ref{IKE4}).

From (\ref{A1}) we get 
\begin{equation*}
g_{M}(CY,\phi U)=0.
\end{equation*}%
From this relation 
\begin{eqnarray*}
0 &=&g_{M}(\nabla _{X}CY,\phi U)+g_{M}(CY,\nabla _{X}\phi U) \\
&&\overset{(\ref{Nambla fi})}{=}g_{M}(\nabla _{X}CY,\phi U)+g_{M}(CY,\phi
\nabla _{X}U) \\
&&\overset{(\ref{3})}{=}g_{M}(\nabla _{X}CY,\phi U)+g_{M}(CY,\phi (\mathcal{A%
}_{X}U)).
\end{eqnarray*}%
Hence we obtain (\ref{IKE6}).
\end{proof}

We now study the integrability of the distribution $(\ker F_{\ast })^{\bot }$
and then we investigate the geometry of leaves of $\ker F_{\ast }$ and $%
(\ker F_{\ast })^{\bot }$.

\begin{theorem}
Let $F$ be an anti-invariant Riemannian submersion from a Sasakian manifold $%
M(\phi ,\xi ,\eta ,g_{M})$ to a Riemannian manifold $(N,g_{N})$. Then the
following assertions are equivalent to each other;
\end{theorem}

$\ i)$ $(\ker F_{\ast })^{\bot }$ \textit{is integrable.}

$ii)$ 
\begin{eqnarray*}
g_{N}((\nabla F_{\ast })(Y,\mathcal{A}_{X}\xi ),F_{\ast }\phi V)
&=&g_{N}((\nabla F_{\ast })(X,\mathcal{A}_{Y}\xi ),F_{\ast }\phi V) \\
&&+g_{M}(CX,\phi \mathcal{A}_{Y}V)-g_{M}(CY,\phi \mathcal{A}_{X}V) \\
&&+g_{M}(\mathcal{A}_{X}\xi ,V)\eta (Y)-g_{M}(\mathcal{A}_{Y}\xi ,V)\eta (X).
\end{eqnarray*}

$iii)$%
\begin{eqnarray*}
g_{M}(\mathcal{A}_{X}\mathcal{A}_{Y}\xi -\mathcal{A}_{Y}\mathcal{A}_{X}\xi
,\phi V) &=&+g_{M}(CX,\phi \mathcal{A}_{Y}V)-g_{M}(CY,\phi \mathcal{A}_{X}V)
\\
&&+g_{M}(\mathcal{A}_{X}\xi ,V)\eta (Y)-g_{M}(\mathcal{A}_{Y}\xi ,V)\eta (X)
\end{eqnarray*}%
\textit{for }$X,Y\in \Gamma ((\ker F_{\ast })^{\bot })$\textit{\ and }$V\in
\Gamma (\ker F_{\ast }).$

\begin{proof}
From (\ref{IREM}), (\ref{Namblafi3}) and (\ref{IKE1}) we have%
\begin{equation}
g_{M}(\nabla _{X}Y,V)=g_{M}(\nabla _{X}CY,\phi V)-g_{M}(\nabla _{X}\mathcal{A%
}_{Y}\xi ,\phi V)+g_{M}(\mathcal{A}_{X}\xi ,V)\eta (Y)  \label{CM1}
\end{equation}%
\textit{for }$X,Y\in \Gamma ((\ker F_{\ast })^{\bot })$\textit{\ and }$V\in
\Gamma (\ker F_{\ast }).$Using (\ref{IKE4}) in (\ref{CM1}) we obtain%
\begin{equation*}
g_{M}(\nabla _{X}Y,V)=g_{M}(\nabla _{X}CY,\phi V)-g_{M}(\mathcal{A}_{Y}\xi
,\phi \mathcal{A}_{X}V)+g_{M}(\mathcal{A}_{X}\xi ,V)\eta (Y)
\end{equation*}%
By help (\ref{IKE4}) and (\ref{IKE6}), the last relation becomes 
\begin{equation*}
g_{M}(\nabla _{X}Y,V)=-g_{M}(CY,\phi \mathcal{A}_{X}V)-g_{M}(\nabla _{X}%
\mathcal{A}_{Y}\xi ,\phi V)+g_{M}(\mathcal{A}_{X}\xi ,V)\eta (Y)
\end{equation*}%
Interchanging the role of $X$ and $Y,$ we get 
\begin{equation*}
g_{M}(\nabla _{Y}X,V)=-g_{M}(CX,\phi \mathcal{A}_{Y}V)-g_{M}(\nabla _{Y}%
\mathcal{A}_{X}\xi ,\phi V)+g_{M}(\mathcal{A}_{Y}\xi ,V)\eta (X)
\end{equation*}%
so that combining the above two relations, we have%
\begin{eqnarray*}
g_{M}([X,Y],V) &=&g_{M}(\nabla _{Y}\mathcal{A}_{X}\xi ,\phi V)-g_{M}(\nabla
_{X}\mathcal{A}_{Y}\xi ,\phi V) \\
&&+g_{M}(CX,\phi \mathcal{A}_{Y}V)-g_{M}(CY,\phi \mathcal{A}_{X}V) \\
&&+g_{M}(\mathcal{A}_{X}\xi ,V)\eta (Y)-g_{M}(\mathcal{A}_{Y}\xi ,V)\eta (X).
\end{eqnarray*}%
Since differential $F_{\ast }$ preserves the lenghts of horizontal vectors
we obtain 
\begin{eqnarray*}
g_{M}([X,Y],V) &=&g_{N}(F_{\ast }\nabla _{Y}\mathcal{A}_{X}\xi ,F_{\ast
}\phi V)-g_{N}(F_{\ast }\nabla _{X}\mathcal{A}_{Y}\xi ,F_{\ast }\phi V) \\
&&+g_{M}(CX,\phi \mathcal{A}_{Y}V)-g_{M}(CY,\phi \mathcal{A}_{X}V) \\
&&+g_{M}(\mathcal{A}_{X}\xi ,V)\eta (Y)-g_{M}(\mathcal{A}_{Y}\xi ,V)\eta (X).
\end{eqnarray*}%
Using (\ref{5}) we have%
\begin{eqnarray*}
g_{M}([X,Y],V) &=&g_{N}(-(\nabla F_{\ast })(Y,\mathcal{A}_{X}\xi ),F_{\ast
}\phi V)-g_{N}(-(\nabla F_{\ast })(X,\mathcal{A}_{Y}\xi ),F_{\ast }\phi V) \\
&&+g_{M}(CX,\phi \mathcal{A}_{Y}V)-g_{M}(CY,\phi \mathcal{A}_{X}V) \\
&&+g_{M}(\mathcal{A}_{X}\xi ,V)\eta (Y)-g_{M}(\mathcal{A}_{Y}\xi ,V)\eta (X).
\end{eqnarray*}%
which proves $(i)\Leftrightarrow (ii).$

On the other hand using (\ref{5}) we get%
\begin{equation*}
(\nabla F_{\ast })(Y,BX)-(\nabla F_{\ast })(X,BY)=-F_{\ast }(\nabla
_{Y}BX-\nabla _{X}BY).
\end{equation*}%
Using(\ref{3}) and (\ref{IKE1}) \ we obtain%
\begin{eqnarray*}
g_{N}(-F_{\ast }(\mathcal{A}_{Y}\mathcal{A}_{X}\xi -\mathcal{A}_{X}\mathcal{A%
}_{Y}\xi ),F_{\ast }\phi V) &=&g_{M}(CX,\phi \mathcal{A}_{Y}V)-g_{M}(CY,\phi 
\mathcal{A}_{X}V) \\
&&+g_{M}(\mathcal{A}_{X}\xi ,V)\eta (Y)-g_{M}(\mathcal{A}_{Y}\xi ,V)\eta (X).
\end{eqnarray*}%
which tells that $(ii)\Leftrightarrow (iii).$
\end{proof}

\begin{remark}
We assume that $(\ker F_{\ast })^{\bot }=\phi \ker F_{\ast }\oplus \{\xi \}.$
Using (\ref{IREM}) one can prove that $CX=0$.
\end{remark}

\begin{theorem}
Let $M(\phi ,\xi ,\eta ,g_{M})$ be a Sasakian manifold \ of dimension $2m+1$
and $(N,g_{N})$ is a Riemannian manifold of dimension $n.$ Let $F:M(\phi
,\xi ,\eta ,g_{M})\rightarrow $ $(N,g_{N})$ be an anti-invariant Riemannian
submersion such that $(\ker F_{\ast })^{\bot }=\phi \ker F_{\ast }\oplus
\{\xi \}.$ Then $(\ker F_{\ast })^{\bot }$ \textit{is not integrable.}
\end{theorem}

\begin{proof}
From (\ref{metric}) it follows that 
\begin{equation*}
\phi (\nabla _{X}Y)=\nabla _{X}BY-g(X,Y)\xi +\eta (Y)X
\end{equation*}%
for $X,Y\in \Gamma ((\ker F_{\ast })^{\bot })$\textit{\ .}Interchanging the
role of $X$ and $Y,$ we get%
\begin{equation*}
\phi (\nabla _{Y}X)=\nabla _{Y}BX-g(X,Y)\xi +\eta (X)Y
\end{equation*}%
so that combining the above two relations, we have%
\begin{equation*}
\phi (\left[ X,Y\right] )=\nabla _{X}BY-\nabla _{Y}BX+\eta (Y)X-\eta (X)Y.
\end{equation*}%
Using (\ref{3}), (\ref{metric}), (\ref{IKE1}) and (\ref{R(xi,,X)Y}) one
obtains%
\begin{equation*}
\phi (\left[ X,Y\right] )=\mathcal{A}_{X}BY-\mathcal{A}_{Y}BX+\mathcal{V}%
\nabla _{X}BY-\mathcal{V}\nabla _{Y}BX+R(X,Y)\xi
\end{equation*}%
If $(\ker F_{\ast })^{\bot }$ is integrable we have%
\begin{equation}
R(X,Y)\xi =\mathcal{A}_{X}\mathcal{A}_{Y}\xi -\mathcal{A}_{Y}\mathcal{A}%
_{X}\xi .  \label{CM2}
\end{equation}

On the otherhand, we know that if $\mathcal{H}=(\ker F_{\ast })^{\bot }$ is
integrable then $\mathcal{A}=0.$ Hence the last relation led to the
contradiction with (\ref{R(xi,,X)Y}).
\end{proof}

From (\ref{4}) and (\ref{nablaXxi}), we can give following Theorem.

\begin{theorem}
Let $M(\phi ,\xi ,\eta ,g_{M})$ be a Sasakian manifold \ of dimension $2m+1$
and $(N,g_{N})$ is a Riemannian manifold of dimension $n.$ Let $F:M(\phi
,\xi ,\eta ,g_{M})\rightarrow $ $(N,g_{N})$ be an anti-invariant Riemannian
submersion such that $\phi (\ker F_{\ast })\subset (\ker F_{\ast })^{\bot }.$
Then $(\ker F_{\ast })^{\bot }$ does not define\textit{\ a totally geodesic
foliation on }$M.$
\end{theorem}

For the distribution $\ker F_{\ast }$, we have;

\begin{theorem}
Let $F$be an anti-invariant Riemannian submersion from a Sasakian manifold $%
M(\phi ,\xi ,\eta ,g_{M})$ to a Riemannian manifold $(N,g_{N})$. Then the
following assertions are equivalent to each other;
\end{theorem}

$\ i)$ $(\ker F_{\ast })$ \textit{defines a totally geodesic foliation on }$%
M.$

$ii)$ $g_{N}(\nabla F_{\ast })(V,\phi X),F_{\ast }\phi W)=0$ \ \ for $X\in
\Gamma ((\ker F_{\ast })^{\bot })$ and $V,W\in \Gamma (\ker F_{\ast }).$

$iii)$ $\mathcal{T}_{V}BX+\mathcal{A}_{CX}V\in \Gamma (\mu ).$

\begin{proof}
Since $g_{M}(W,X)=0$ we have $g_{M}(\nabla _{V}W,X)=-g(W,\nabla _{V}X).$
From (\ref{metric}) and (\ref{IREM}) we get $g_{M}(\nabla
_{V}W,X)=-g_{M}(\phi W,\nabla _{V}BX)-g_{M}(\phi W,\nabla _{V}CX).$ Using (%
\ref{1}) and (\ref{2}) we obtain $g_{M}(\nabla _{V}W,X)=-g_{M}(\phi W,\nabla
_{V}\phi X).$ Then Riemannian submersion $F$ and (\ref{5}) imply that 
\begin{equation*}
g_{M}(\nabla _{V}W,X)=g_{N}(F_{\ast }\phi W,(\nabla F_{\ast })(V,\phi X))
\end{equation*}%
which is $(i)\Leftrightarrow (ii).$ By direct calculation, we derive%
\begin{equation*}
g_{N}(F_{\ast }\phi W,(\nabla F_{\ast })(V,\phi X))=-g_{M}(\phi W,\nabla
_{V}\phi X).
\end{equation*}%
Using (\ref{IREM})we have 
\begin{equation*}
g_{N}(F_{\ast }\phi W,(\nabla F_{\ast })(V,\phi X))=-g_{M}(\phi W,\nabla
_{V}BX+\nabla _{V}CX).
\end{equation*}%
Hence we get%
\begin{equation*}
g_{N}(F_{\ast }\phi W,(\nabla F_{\ast })(V,\phi X))=-g_{M}(\phi W,\nabla
_{V}BX+\left[ V,CX\right] +\nabla _{CX}V).
\end{equation*}%
Since $\left[ V,CX\right] \in \Gamma (\ker F_{\ast }),$ using (\ref{1}) and (%
\ref{3}), we obtain%
\begin{equation*}
g_{N}(F_{\ast }\phi W,(\nabla F_{\ast })(V,\phi X))=-g_{M}(\phi W,\mathcal{T}%
_{V}BX+\mathcal{A}_{CX}V).
\end{equation*}%
This shows $(ii)\Leftrightarrow (iii).$
\end{proof}

\begin{corollary}
Let $F:M(\phi ,\xi ,\eta ,g_{M})\rightarrow $ $(N,g_{N})$ be an
anti-invariant Riemannian submersion such that $(\ker F_{\ast })^{\bot
}=\phi \ker F_{\ast }\oplus \{\xi \},$ where $M(\phi ,\xi ,\eta ,g_{M})$ is
a Sasakian manifold and $(N,g_{N})$ is a Riemannian manifold. Then following
assertions are equivalent to each other;

$i)(\ker F_{\ast })$ \textit{defines a totally geodesic foliation on }$M.$

$ii)(\nabla F_{\ast })(V,\phi X)=0,$ for $X\in \Gamma ((\ker F_{\ast
})^{\bot })$ and $V,W\in \Gamma (\ker F_{\ast }).$

$iii)$ $\mathcal{T}_{V}\phi W=0.$
\end{corollary}

\begin{theorem}
Let $F:M(\phi ,\xi ,\eta ,g_{M})\rightarrow $ $(N,g_{N})$ be an
anti-invariant Riemannian submersion such that $(\ker F_{\ast })^{\bot
}=\phi \ker F_{\ast }\oplus \{\xi \},$ where $M(\phi ,\xi ,\eta ,g_{M})$ is
a Sasakian manifold and $(N,g_{N})$ is a Riemannian manifold. Then $F$ is a
totally geodesic map if and only if%
\begin{equation}
\mathcal{T}_{W}\phi V=0,\text{ \ \ }\forall \text{ }W,\text{ }V\in \Gamma
(\ker F_{\ast })  \label{S1}
\end{equation}%
and%
\begin{equation}
\mathcal{A}_{X}\phi W=0,\text{ \ }\forall X\in \Gamma ((\ker F_{\ast
})^{\bot }),\forall \text{ }W\in \Gamma (\ker F_{\ast }).\text{\ }
\label{S2}
\end{equation}
\end{theorem}

\begin{proof}
First of all, we recall that the second fundamental form of a Riemannian
submersion satisfies (\ref{5a}). For $W,$ $V\in \Gamma (\ker F_{\ast }),$ by
using (\ref{2}), (\ref{5}) and (\ref{Nambla fi}) we get%
\begin{equation}
(\nabla F_{\ast })(W,V)=F_{\ast }(\phi \mathcal{T}_{W}\phi V).  \label{S3}
\end{equation}%
On the other hand by using (\ref{5}) and (\ref{Nambla fi}) we have%
\begin{equation*}
(\nabla F_{\ast })(X,W)=F_{\ast }(\phi \nabla _{X}\phi W)
\end{equation*}%
for $X\in \Gamma ((\ker F_{\ast })^{\bot }).$ Then from (\ref{4}), we obtain 
\begin{equation}
(\nabla F_{\ast })(X,W)=F_{\ast }(\phi \mathcal{A}_{X}\phi W-g(W,\phi X)\xi
).  \label{S4}
\end{equation}%
Since $\phi $ is non-singular, proof comes from (\ref{S3}), (\ref{S4}) and (%
\ref{5a}).
\end{proof}

Finally, we give a necessary and sufficient condition for an anti-invariant
Riemannian submersion such that $(\ker F_{\ast })^{\bot }=\phi \ker F_{\ast
}\oplus \{\xi \}$ to be harmonic.

\begin{theorem}
Let $F:M(\phi ,\xi ,\eta ,g_{M})\rightarrow $ $(N,g_{N})$ be an
anti-invariant Riemannian submersion such that $(\ker F_{\ast })^{\bot
}=\phi \ker F_{\ast }\oplus \{\xi \},$ where $M(\phi ,\xi ,\eta ,g_{M})$ is
a Sasakian manifold and $(N,g_{N})$ is a Riemannian manifold. Then $F$ is
harmonic if and only if Trace$\phi \mathcal{T}_{V}=0$ for $V\in \Gamma (\ker
F_{\ast }).$
\end{theorem}

\begin{proof}
From \cite{EJ} we know that $F$ is harmonic if and only if $F$ has minimal
fibres. Thus $F$ is harmonic if and only if $\dsum\limits_{i=1}^{k}\mathcal{T%
}_{e_{i}}e_{i}=0,$ where $k$ is dimension of $\ker F_{\ast }$ . On the other
hand, from (\ref{1}), (\ref{2}) and (\ref{Nambla fi}) we get%
\begin{equation}
\mathcal{T}_{V}\phi W=\phi \mathcal{T}_{V}W  \label{S5}
\end{equation}%
for any $W,$ $V\in \Gamma (\ker F_{\ast }).$ Using (\ref{S5}), we get 
\begin{equation*}
\dsum\limits_{i=1}^{k}g_{M}(\mathcal{T}_{e_{i}}\phi
e_{i},V)=-\dsum\limits_{i=1}^{k}g_{M}(\mathcal{T}_{e_{i}}e_{i},\phi V)
\end{equation*}%
for any $V\in \Gamma (\ker F_{\ast }).$ (\ref{4b}) implies that%
\begin{equation*}
\dsum\limits_{i=1}^{k}g_{M}(\phi e_{i},\mathcal{T}_{e_{i}}V)=\dsum%
\limits_{i=1}^{k}g_{M}(\mathcal{T}_{e_{i}}e_{i},\phi V).
\end{equation*}%
Then, using (\ref{TUW}) we have%
\begin{equation*}
\dsum\limits_{i=1}^{k}g_{M}(\phi e_{i},\mathcal{T}_{V\text{ }%
}e_{i})=\dsum\limits_{i=1}^{k}g_{M}(\mathcal{T}_{e_{i}}e_{i},\phi V).
\end{equation*}%
Hence, proof comes from (\ref{metric}).
\end{proof}

\end{document}